\documentclass{birkjour}

\newtheorem{thm}{Theorem}[section]
\theoremstyle{definition}
\newtheorem{defn}[thm]{Definition}
\theoremstyle{remark}
\newtheorem{rem}[thm]{Remark}

\numberwithin{equation}{section}

\AtBeginDocument{{\noindent\small
This is a preprint of a paper whose final and definite form is with 
\emph{Chaos, Solitons \& Fractals}, ISSN: 0960-0779. 
Submitted 16 Dec 2016; Article revised 17 Apr 2017;
Article accepted for publication 19 Apr 2017;
DOI: 10.1016/j.chaos.2017.04.035.} 
\vspace{9mm}}

\begin{document}

\title[Fractional Herglotz Variational Principles]{Fractional 
Herglotz Variational Principles\\
with Generalized Caputo Derivatives}
  
\author[R. Garra]{Roberto Garra}
\address{Dipartimento di Scienze Statistiche, 
``Sapienza'' Universit\`a di Roma, Rome, Italy}
\email{roberto.garra@sbai.uniroma1.it}

\author[G. S. Taverna]{Giorgio S. Taverna}
\address{Institute for Climate and Atmospheric Science, 
University of Leeds, Leeds, UK}
\email{eegst@leeds.ac.uk}
                
\author[D. F. M. Torres]{Delfim F. M. Torres}
\address{Center for Research and Development in Mathematics and Applications (CIDMA),\\
Department of Mathematics, University of Aveiro, 3810-193 Aveiro, Portugal}
\email{delfim@ua.pt}




\date{Submitted: December 16, 2016}

    
\begin{abstract}
We obtain Euler--Lagrange equations,
transversality conditions and a Noether-like
theorem for Herglotz-type variational problems 
with Lagrangians depending on generalized 
fractional derivatives. As an application, 
we consider a damped harmonic oscillator 
with time-depending mass and elasticity, 
and arbitrary memory effects.
\end{abstract}


\maketitle


\section{Introduction}

Fractional variational principles and their applications
is a subject under strong current research 
\cite{MR3443073,MR3331286,book}.
For classical fields with fractional derivatives, 
by using the fractional Lagrangian formulation,
we can refer to \cite{MR2163462}. An Hamiltonian approach to
fractional problems of the calculus of variations is given
in \cite{MR2279972}, where the Hamilton equations of motion 
are obtained in a manner similar to the one found in 
classical mechanics. In addition, classical fields with fractional 
derivatives are investigated using the Hamiltonian formalism \cite{MR2279972}. 
A method for finding fractional Euler--Lagrange equations with 
Caputo derivatives, by making use of a fractional generalization of the classical Fa\'a 
di Bruno formula, can be found in \cite{Baleanu=Referee:asks:self:citation}.
There the fractional Euler--Lagrange and Hamilton equations are obtained 
within the so called $1 + 1$ field formalism \cite{Baleanu=Referee:asks:self:citation}.
For discrete versions of fractional derivatives with a nonsingular 
Mittag-Leffler function see \cite{MR3544397}, where the properties of such fractional 
differences are studied and discrete integration by parts formulas proved in order 
to obtain Euler--Lagrange equations for discrete variational problems \cite{MR3544397}. 
The readers interested in the discrete fractional calculus of variations are refereed 
to the pioneer work of Bastos et al. \cite{MR2728463,MyID:179}. Here we are interested 
in the generalized continuous calculus of variations introduced by Herglotz.

The generalized variational principle firstly proposed 
by Gustav Herglotz in 1930 \cite{her} gives a variational 
principle description of non-conservative systems even when 
the Lagrangian is autonomous \cite{MyID:281,MyID:319}.
It is essentially based on the following problem:
find the trajectories $x(t)$, satisfying given boundary conditions,
that extremize (minimize or maximize) the terminal value $z(b)$ 
of the functional $z$ that satisfies the differential equation
\begin{equation*}
\dot{z}(t)= L\left(t,x(t),\dot{x}(t),z(t)\right), \quad t\in [a,b],
\end{equation*}
subject to the initial condition $z(a)= \gamma$.
Herglotz proved that the necessary condition for a trajectory 
to be an extremizer of the generalized variational problem 
is to satisfy the generalized Euler--Lagrange equation
\begin{equation*}
\frac{\partial L}{\partial x}-\frac{d}{dt}
\frac{\partial L}{\partial \dot{x}}+\frac{\partial L}{\partial z}
\frac{\partial L}{\partial \dot{x}}=0.
\end{equation*} 
The main physical motivation for the development of generalized 
variational methods behind the classical calculus of variations
is linked to the inverse variational problem of classical mechanics 
in the cases where dissipation is not negligible \cite{inverse}. 
In \cite{almeida}, Almeida and Malinowska have considered 
a fractional variational Herglotz principle, 
where fractionality stands in the dependence of the Lagrangian 
by the Caputo fractional derivative of the generalized variables.
See also the more recent papers \cite{MyID:369,MyID:368}
on fractional Herglotz  variational principles.
In \cite{MyID:369}, new necessary conditions for higher-order 
generalized variational problems with time delay, which 
are semi-invariant under a group of transformations 
that depends on arbitrary functions, is obtained.
Fractional variational problems of Herglotz type of 
variable order are investigated in \cite{MyID:368},
where necessary optimality conditions, described by
fractional differential equations depending on 
a combined Caputo fractional derivative 
of variable order, are proved, both for one 
and several independent variables \cite{MyID:368}. 
Using such results, it is possible to find, by using a variational approach, 
the equations of motion of a dissipative mechanical system 
with memory. This is useful, since in many classical cases, 
memory effects play a relevant role in systems with dissipation.
A relevant example is given by the Basset memory force acting 
on a sphere rotating in a Stokes fluid. It is well-known that
this force can be represented by means of Caputo derivatives 
of order $1/2$ (see, e.g., \cite{noi} and references therein).
However, a limit in this approach stands to the fact that 
the particular choice of the dependence of the Lagrangian
by the Caputo fractional derivative of the generalized variable, 
implies that it describes equations of motion of systems with 
power-law memory kernels. On the other hand, in the framework 
of the fractional calculus of variations (a quite recent topic 
of research, started from the seminal investigations of Riewe 
\cite{Riewe} and then developed by many researchers, see for 
example the recent monographs 
\cite{MR3443073,book1,book:Klimek,MR3331286,book}), 
Odzijewicz et al. \cite{mali,MyID:268} discusses the case of the 
Lagrangian depending on generalized Caputo-type derivatives 
with arbitrary completely monotonic kernels \cite{MR3331286}. 
Here we consider a generalized calculus of variations 
in the sense of Herglotz with Lagrangians depending on 
generalized Caputo-type operators treated 
in \cite{MR3200762,MR3331286,mali}. Our aim is to find 
a general and adequate variational approach 
to describe mechanical systems with arbitrary memory forces.

The paper is organized as follows. In Section~\ref{sec:2},
we recall the necessary definitions and results
from the generalized fractional variational calculus.
Our results are then given in Sections~\ref{sec:3},
\ref{sec:4} and \ref{sec:5}: we prove in Section~\ref{sec:3}
necessary optimality conditions of Euler--Lagrange
type (Theorem~\ref{thm:EL}) and transversality conditions
(Theorem~\ref{thm:TC}) to the generalized 
fractional variational problem of Herglotz;
we show in Section~\ref{sec:4} how our approach can deal, 
in an elegant way, with dissipative dynamical systems 
with memory effects and time-varying mass and elasticity;
and we obtain a generalized fractional Herglotz Noether theorem
(Theorem~\ref{thm:Noether}) in Section~\ref{sec:5}.
We end with Section~\ref{sec:6} of conclusions
and some directions of future work.

        
\section{Preliminaries}
\label{sec:2}

In this section, we recall the main definitions of the 
generalized Riemann--Liouville and Caputo-like operators
and their properties, according to the analysis of
generalized fractional variational principles developed 
in \cite{MR3331286,mali}. For a general introduction 
to fractional differential operators and equations 
we refer to the classical encyclopedic book \cite{samko}. 
See also \cite{MR1658022}. For an introduction 
to the fractional variational methods, and in particular 
integration by parts formulas for fractional 
integrals and derivatives, we refer to the monographs 
\cite{book1,book}. For computational and numerical 
aspects see \cite{MR3443073,MR2838157,MyID:339}.

\begin{defn}
The operator $K_P^\alpha$ is given by
\begin{equation*}
\begin{split}
K_P^\alpha[f](x)&= K_P^\alpha[t\rightarrow f(t)](x)\\
&= p\int_a^x k_\alpha (x,t) f(t)dt
+ q\int_x^b k_\alpha (t,x) f(t)dt,
\end{split}
\end{equation*}
where $P= \langle a,x,b,p,q \rangle$ is the parameter set, 
$x\in [a,b]$, $p,q\in \mathbb{R}$, and $k_\alpha(x,t)$ is
a completely monotonic kernel. 
\end{defn}

For the sake of completeness, we should remark that similar 
generalizations of the Riemann--Liouville integrals have been 
considered in the framework of the fractional action-like 
variational approach (FALVA) \cite{MR2421931,fizika}. 
On the other hand, a similar generalization 
is considered in a probabilistic framework
in \cite{Toaldo}.

\begin{thm}[See \protect{\cite[Theorem~3]{mali}}]
Let $k_\alpha \in L_1([a,b])$ and 
$$
k_\alpha(x,t)= k_\alpha(x-t).
$$ 
Then, the operator $K_P^\alpha:L_1([a,b])\rightarrow L_1([a,b])$
is a well-defined bounded and linear operator. 
\end{thm}
        
\begin{defn}
Let $P$ be a given parameter set and $\alpha\in (0,1)$. 
The operator $A^\alpha = D\circ K_P^{1-\alpha}$ is 
the generalized Riemann--Liouville derivative, where
$D$ stands for the conventional integer order derivative operator. 
\end{defn}

The corresponding generalized Caputo derivative is defined as
$B_P^\alpha = K_P^{1-\alpha}\circ D$. A key-role in the following 
analysis is played by the following theorem that provides 
the integration by parts formula for the generalized operators defined before.

\begin{thm}[See \protect{\cite[Theorem 11]{mali}}]
\label{thm:ibp}
Let $\alpha \in (0,1)$ and $P= \langle a, x, b, p, q \rangle$.
If $f, K_P^{1-\alpha}g \in AC([a,b])$, where 
$P^*= \langle a, x, b, q,p \rangle$ and $k_\alpha(\cdot)$ 
is a square-integrable function on $\Delta = [a,b]\times [a,b]$, then 
\begin{equation*}
\int_a^b g(x) B_P^\alpha[f](x)dx
= f(x)K_{P^*}^{1-\alpha}[g](x)\bigg|^a_b
-\int_a^b f(x) A_{P^*}^\alpha[g](x)dx.
\end{equation*}
\end{thm}
        
        
\section{Generalized fractional Herglotz variational principles}
\label{sec:3}
        
One of the main aims of this work is to prove generalized Euler--Lagrange 
equations related to the generalized fractional variational 
principle of Herglotz. In particular, the generalization 
is based on the fact that the Lagrangian depends 
on the generalized Caputo derivative $B_P^\alpha$. 
As explained before, by using this approach, 
we will be able to find a variational approach to mechanical systems 
involving an arbitrary (suitable) memory kernel $k_\alpha (t)$.
Therefore, let us consider the differential equation
\begin{equation}
\label{z}
\dot{z}(t)= L\left(t,x(t), B_P^\alpha[x](t),z(t)\right), 
\quad t\in [a,b],
\end{equation}
with the initial condition $z(a)= z_a$. We moreover assume that 
\begin{itemize}
\item $x(a)= x_a$, $x(b)= x_b$, $x_a,x_b\in \mathbb{R}^n$,
\item $\alpha \in (0,1)$,
\item $x\in C^1([a,b],\mathbb{R}^n)$, 
$B_P^\alpha[x]\in C^1([a,b],\mathbb{R}^n)$,
\item the Lagrangian $L:[a,b]\times \mathbb{R}^{2n+1}
\rightarrow \mathbb{R}$ is of class $C^1$ and the maps 
$t\rightarrow \lambda(t)\frac{\partial L}{\partial B_P^\alpha x_j}[x,z](t)$
exist and are continuous on $[a,b]$, where we use the notations
\begin{gather*}
[x,z](t):=(t,x(t),B_P^\alpha[x](t),z(t)),\\ 
\lambda(t)= \exp\left(
-\int_a^t \frac{\partial L}{\partial z}[x,z](\tau)d\tau \right).
\end{gather*}
\end{itemize}

The generalized fractional Herglotz variational principle 
is formulated as follows:
\begin{quote}
\textit{Let functional $z(t)= z[x;t]$ be given by the differential 
equation \eqref{z} and $\eta\in C^1([a,b], \mathbb{R})$
be an arbitrary function such that $\eta(a)=\eta(b)=0$
and $B_P^\alpha[\eta]\in C^1([a,b],\mathbb{R})$.
Then, the value of the functional $z$ has an extremum 
for the function $x$ if and only if
\begin{equation*}
\left.\frac{d}{d\epsilon}z[x+\epsilon \eta;b]\right|_{\epsilon =0}=0.
\end{equation*}}
\end{quote}     		

We are now able to state and prove a generalized fractional
necessary optimality condition of Euler--Lagrange type that,
together with the fractional Herglotz Noether theorem
(see Theorem~\ref{thm:Noether}), constitute the central 
results of the paper.

\begin{thm}[Generalized fractional Herglotz Euler--Lagrange equations]
\label{thm:EL}
Let function $x$ be such that $z[x;b]$ in \eqref{z} attains an extremum.
Then $x(t)$, $t \in [a,b]$, is a solution to the generalized 
Euler--Lagrange equations
\begin{equation}
\label{elg}
\lambda(t)\frac{\partial L}{\partial x_j}[x,z](t)+A^{\alpha}_{P^*}
\left(\lambda(t)\frac{\partial L}{\partial B_P^\alpha x_j}[x,z](t)\right)= 0,
\end{equation}
$j = 1, \ldots, n$.
\end{thm}

\begin{proof}
The proof uses the generalized integration by parts 
formula for Caputo-like operators $B_P^\alpha$ given
by Theorem~\ref{thm:ibp}. Let $x$ be such that the functional 
$z[x;b]$ attains an extremum. The rate of change of $z$ 
in the direction $\eta$ is given by 
\begin{equation*}
\theta(t) = \frac{d}{d\epsilon}z[x+\epsilon \eta;t]\bigg|_{\epsilon=0}. 
\end{equation*}
The variation $\epsilon \eta $ of the argument 
in equation \eqref{z} is given by
\begin{equation*}
\frac{d}{dt}z[x+\epsilon \eta;t]
= L\left(t, x(t)+\epsilon \eta(t), B_P^\alpha[x](t)
+\epsilon B_P^\alpha[\eta](t), z[x+\epsilon \eta; t]\right).
\end{equation*}
Observing that, from equation \eqref{z}, we have
\begin{equation*}
\frac{d}{dt}\theta(t) 
= \frac{d}{d\epsilon}L\bigg(t,x(t)+\epsilon \eta(t),
\left. B_P^\alpha[x](t)+\epsilon B_P^\alpha[\eta](t), 
z[x+\epsilon \eta;t]\bigg)\right|_{\epsilon = 0}, 	
\end{equation*}
this gives a differential equation of the form 
\begin{equation*}
\frac{d\theta(t)}{dt}-\frac{\partial L}{\partial z}[x,z](t)\theta(t)
= \sum_{j=1}^n \left(\frac{\partial L}{\partial x_j}[x,z](t)\eta_j(t)
+\frac{\partial L}{\partial B^{\alpha}_P x_j}B_P^{\alpha}[\eta_j](t)\right),
\end{equation*}
whose solution is 
\begin{equation*}
\int_a^t \sum_{j=1}^n \left(\frac{\partial L}{\partial x_j}[x,z](t)\eta_j(t)
+\frac{\partial L}{\partial B^{\alpha}_P x_j}
B_P^{\alpha}[\eta_j](t)\right)\lambda(t)dt
=\theta(t)\lambda(t)-\theta(a),
\end{equation*}
where $\lambda(t)= \exp\left(-\int_a^t 
\frac{\partial L}{\partial z}[x,z](\tau)d\tau\right)$ 
and $\theta(a)=0$. For $t= b$, we get
\begin{equation*}
\int_a^b \sum_{j=1}^n \left(\frac{\partial L}{\partial x_j}[x,z](t)\eta_j(t)
+\frac{\partial L}{\partial B^{\alpha}_P x_j}
B_P^{\alpha}[\eta_j](t)\right)\lambda(t)dt
= \theta(b)\lambda(t).
\end{equation*}
Since $\theta(b)$ is the variation of $z[x;b]$, if $x$ 
gives a maximum, also $\theta(b)=0$, and therefore we get 
\begin{equation*}
\int_a^b\sum_{j=1}^n \left(\frac{\partial L}{\partial x_j}[x,z](t)
\eta_j(t)+\frac{\partial L}{\partial B^{\alpha}_P x_j}
B_P^{\alpha}[\eta_j](t)\right)\lambda(t)dt= 0.
\end{equation*}
Then, by using the integration by parts formula 
(see Theorem~\ref{thm:ibp}), and the fact that 
$\eta(a)= \eta(b)=0$, we obtain the claimed result
by the fundamental lemma of the calculus of variations.
\end{proof}

\begin{rem}
Let $k_\alpha(x,t) = \frac{1}{\Gamma(1-\alpha)} (x-t)^{-\alpha}$,
$\alpha \in (0,1)$, and the parameter set be given by
$P= \langle a,x,b,1,0 \rangle$. In this particular case,
the operator $B_P^{\alpha}$ coincides with the standard
Caputo fractional derivative $_a^CD_x^\alpha$,
and our Theorem~\ref{thm:EL} gives the Euler--Lagrange
equation of \cite{almeida}.
\end{rem}
     	
Observe that, in order to determine uniquely the unknown function 
that satisfies equation \eqref{z}, the following system 
of differential equations
\begin{equation}
\label{z1}
\begin{cases}
\dot{z}(t)= L[x,z](t),\\
\displaystyle{\lambda(t)\frac{\partial L}{\partial x_j}[x,z](t)
+A^{\alpha}_{P^*}\left(\lambda(t)\frac{\partial L}{\partial 
B_P^\alpha x_j}[x,z](t)\right)= 0}, 
\end{cases}
\end{equation}
$j = 1, \ldots, n$, should be studied with the given boundary conditions. 
In the case $x_j(b)$ is not fixed, similar arguments as those used 
in the proof of Theorem~\ref{thm:EL} allow us to obtain the generalized 
fractional integral transversality conditions \eqref{eq:TC}. 

\begin{thm}[Generalized fractional Herglotz transversality conditions]
\label{thm:TC}
Let $x$ be such that $z(b)=z[x;b]$ in equation \eqref{z} attains an extremum. 
Then $x$ is a solution to the system \eqref{z1}. Moreover, if $x_j(b)$ 
is not fixed, $j \in \{ 1, \dots, n\}$, 
then the integral transversality condition
\begin{equation}	
\label{eq:TC}
K_{P^*}^{1-\alpha}\left[t\rightarrow \lambda(t)
\frac{\partial L}{\partial B_P^{\alpha}x_j}[x,z](t)\right](b)= 0
\end{equation}
holds. 
\end{thm}


\section{An application: generalized damped harmonic oscillator with memory effects}
\label{sec:4}

As already discussed in the literature, generalized variational 
methods can be useful to treat the inverse problem for dissipative 
systems where memory or damping effects are not negligible. 
For example, in \cite{mali}, the authors discussed an application 
of generalized variational problems to the Caldirola--Kanai approach 
to quantum dissipative systems; while in \cite{noi} an application 
to the inverse problem for the Basset system was discussed. 
One can argue that the starting point to research on 
generalized variational problems was given by the Lemma of Bauer 
\cite{Bauer}, stating that the equations of motion 
of a classical dissipative system with constant coefficients 
cannot be derived from a classical variational approach. 
Here we show the peculiarity of the approach considered in this paper, 
to treat dissipative dynamical systems with memory effects 
and time-varying mass and elasticity. Let us consider 
the mechanical system described by the following autonomous Lagrangian:
\begin{equation}
\label{eq:Lag:appl}
L\left(x(t), B_P^\alpha[x](t),z(t)\right) 
= \frac{1}{2}m\left( B_P^\alpha[x](t) \right)^2 
- \frac{1}{2}kx^2(t) + \lambda_0 z(t),
\end{equation}
where we take $a=0$, $b>0$, $p=1$ and $q= 0$.
By using the generalized Euler--Lagrange equations discussed 
in Theorem \ref{thm:EL}, we have that the equation 
of motion in this case is given by
\begin{equation}
\label{dho}
m A_{P^*}^\alpha\left(e^{-\lambda_0 t}
\left( B_P^\alpha[x](t) \right)\right)
- k e^{-\lambda_0 t}x(t) = 0.
\end{equation}
Equation \eqref{dho} describes a dynamical oscillatory system 
with exponentially time-decaying mass and elasticity 
coefficient and arbitrary memory. The generalized \textit{velocity} 
$B_P^\alpha[x](t)$ can be physically interpreted as a time--weighted 
velocity, where memory effects are induced by viscosity 
(and clearly depends by the relaxation kernel $k_\alpha$ 
in the definition of the generalized operators $B_P^\alpha$ 
and $A_{P^*}^\alpha$). Therefore, the generalized Herglotz approach 
here considered provides a variational method to treat 
the inverse problem of a damped harmonic oscillator 
with time-depending mass and elasticity (as a consequence 
of the Herglotz approach) and with arbitrary memory effect 
implying a velocity delay (due to the dependence of the Lagrangian 
by a generalized fractional integral). Clearly, in the case 
in which the memory is neglected, that is, $\alpha \rightarrow 1$,
the Lagrangian \eqref{eq:Lag:appl} takes the form
\begin{equation}
\label{eq:L:a1}
L\left(x(t),\dot{x}(t),z(t)\right)
= \frac{1}{2}m\dot{x}^2(t) - \frac{1}{2}k x^2(t)+\lambda_0 z(t)
\end{equation}
and we recover from \eqref{dho} the equation 
$m \frac{d}{dt}\left(e^{-\lambda_0 t}\dot{x}(t) \right)
+k e^{-\lambda_0 t}x(t) = 0$
of an harmonic oscillator with dissipation. 
Moreover, if $\lambda_0 = 0$, then we are in the case in which 
the Lagrangian \eqref{eq:L:a1} does not depend on $z(t)$ 
and we have the classical equation of the harmonic oscillator:
if $\alpha \rightarrow 1$ and $\lambda_0 = 0$, then
the Euler--Lagrange equation \eqref{dho} reduces 
to $m \ddot{x}(t) + k x(t) = 0$.


\section{Noether theorem for generalized 
fractional Herglotz variational problems} 
\label{sec:5}
 
The analysis of possible generalizations of Noether-type theorems 
to fractional and Herglotz variational principles, has been subject 
of recent research: we refer, for example, to \cite{MyID:319,MyID:327} 
and the references therein. Noether-like theorems play indeed 
a key-role in mathematical-physics by giving the relation between 
the invariance of the action with respect to some parametric 
transformation and the existence of conserved quantities
\cite{MR1980565,MR2099056}. Here we consider a one-parameter 
family of transformations
\begin{equation}
\label{tra}
\bar{x}_j = h_j(t,x,s), \quad j = 1,\ldots, n,
\end{equation}
depending on the parameter $s\in (-\epsilon,+\epsilon)$, 
with $h_j\in C^2$ such that 
$$
h_j(t,x,0)= x_j \text{ for all } (t,x)\in [a,b]\times \mathbb{R}^n.
$$
The Taylor expansion is given by 
\begin{equation*}
\begin{split}
h_j(t,x,\delta)
&= h_j(t,x,0)+s\xi_j(t,x)+o(s)\\
&= x_j+s\xi_j(t,x)+o(s),
\end{split}
\end{equation*}
where $\xi_j(t,x)= \partial_s h_j(t,x,s)\bigg|_{s=0}$.
In this case the linear approximation of the transformation 
\eqref{tra} is simply given by 
\begin{equation}
\label{trans}
\bar{x}_j(t)= h_j(t,x(t),s), \quad j=1, \ldots, n.
\end{equation}
Let 
$$
\theta(t) = \frac{d}{ds}\bar{z}[x+s\xi,t]|_{s=0}
$$
be the total variation produced by the transformation \eqref{tra}. 

\begin{defn}
\label{def:inv}
The transformation $\bar{x}$ given by \eqref{tra} 
leaves the functional $z$ \eqref{z} invariant if 
$\theta(t) \equiv 0$.
\end{defn}

By using the generalized fractional Euler--Lagrange equation \eqref{elg}
of Theorem~\ref{thm:EL}, we are now able to state the following 
Noether-type result.

\begin{thm}[Generalized fractional Herglotz Noether theorem]
\label{thm:Noether}
If the functional $z$ in \eqref{z} is invariant in the sense 
of Definition~\ref{def:inv}, then
\begin{equation*}
\sum_{j=1}^n \mathcal{\hat{O}}^{\alpha}\bigg[
\lambda(t)\frac{\partial L}{\partial B_P^{\alpha}x_j}[x,z](t), 
\xi_j(t,x(t))\bigg]=0
\end{equation*}
holds along the solutions of the generalized Euler--Lagrange 
equation \eqref{elg}, where 
$$
\displaystyle \lambda(t)
= \exp\left(-\int_0^t \frac{\partial L}{\partial z}[x,z](\tau)
d\tau\right)
$$
and
$$
\mathcal{\hat{O}}^\alpha [f,g]:=fB_P^\alpha g- gA_{P^*}^\alpha f.
$$
\end{thm}   
  
\begin{proof}
By using the transformation \eqref{trans} 
into equation \eqref{z}, we get 
\begin{equation*}
\frac{d}{dt}\bar{z}(t) = L\left(t,\bar{x}(t), 
B_P^\alpha \bar{x}(t), \bar{z}(t)\right).
\end{equation*}
Differentiating with respect to $s$ and setting $s=0$, we obtain
\begin{equation}
\label{due}
\dot{\theta}(t)-\frac{\partial L}{\partial z} \theta(t)
= \sum_{j=1}^n\left(\frac{\partial L}{\partial x_j}\xi_j 
+\frac{\partial L}{\partial B_P^{\alpha} x_j}B_P^{\alpha}\xi_j\right),
\end{equation}
where we omit, in order to simplify the notation, 
that the partial derivatives of $L$ are evaluated at 
$[x,z](t)$ and $\xi_j$ at $(t,x(t))$. The solution 
of \eqref{due} is given by
\begin{equation}
\label{aa}
\theta(t)\lambda(t)-\theta(a)
= \int_a^t\sum_{j=1}^n\left(\frac{\partial L}{\partial x_j}\xi_j 
+\frac{\partial L}{\partial B_P^{\alpha} x_j}
B_P^{\alpha}\xi_j\right)\lambda(\tau)d\tau.
\end{equation}
Since, along the solutions of the generalized Euler--Lagrange 
equation \eqref{elg}, we have that
\begin{equation}
\label{eq:lambdaLxj}
\lambda(t)\frac{\partial L}{\partial x_j}
= -A_{P^*}^{\alpha}\left(\lambda(t)\frac{\partial L}{\partial 
B_P^{\alpha} x_j}\right),
\end{equation}
we obtain the claimed result
by substituting \eqref{eq:lambdaLxj} into \eqref{aa}.
\end{proof} 	

        
\section{Conclusions}
\label{sec:6}       

We introduced the study
of fractional variational problems of Herglotz type
that depend on generalized fractional operators
in the sense of \cite{MR3331286,MyID:288}. 
As a particular case, one gets a generalization 
of the Herglotz variational principle 
for non-conservative  systems with Caputo derivatives.
Main results give a necessary optimality condition of Euler--Lagrange
type (Theorem~\ref{thm:EL}), integral transversality 
conditions (Theorem~\ref{thm:TC}), and a Noether-type theorem
(Theorem~\ref{thm:Noether}). Our motivation comes 
from physics, where such variational principles can be 
used to describe mechanical systems with memory of arbitrary form. 
As an application, a fractional mechanical system is analyzed  
with a fractionally generalized velocity that reproduces, for $\alpha = 1$,  
the standard Lagrangian of a harmonic oscillator with exponential damping,
which also contains the non-damped conservative oscillator. 

The research here initiated can now be enriched in different directions, 
by trying to bring to the fractional setting the recent results  
\cite{MyID:281,MyID:313,MyID:319,MyID:327,MyID:342}
of Santos et al. on Herglotz variational problems. 

        
\subsection*{Acknowledgments}

Torres was supported by Portuguese funds through CIDMA and FCT,
within project UID/MAT/04106/2013. The authors are grateful to 
two anonymous referees for valuable comments and suggestions, 
which helped them to improve the quality of the paper. 



\end{document}